\documentclass[11pt]{amsart}

\usepackage[T1]{fontenc}
\usepackage{lmodern}
\usepackage[a4paper,margin=1.05in]{geometry}
\usepackage{amsmath,amssymb,amsthm,mathtools}
\usepackage{booktabs}
\usepackage{tabularx}
\usepackage{xcolor}
\usepackage{hyperref}

\hypersetup{
  colorlinks=true,
  linkcolor=blue!55!black,
  citecolor=red!40!black,
  urlcolor=blue!65!black
}

\numberwithin{equation}{section}

\theoremstyle{plain}
\newtheorem{theorem}{Theorem}[section]
\newtheorem{proposition}[theorem]{Proposition}
\newtheorem{lemma}[theorem]{Lemma}
\newtheorem{corollary}[theorem]{Corollary}

\theoremstyle{remark}
\newtheorem*{remark}{Remark}
   
\newcommand{\diver}{\operatorname{div}}
\newcommand{\curl}{\operatorname{curl}}
\newcommand{\supp}{\operatorname{supp}}
\newcommand{\cJ}{\mathcal J}
\newcommand{\norm}[2]{\left\lVert #1\right\rVert_{#2}}

\usepackage{nicefrac}
\newcommand*{\bydef}{\overset{\rm def}{=}}
\usepackage{todonotes}  

\title[EM systems: between well-posedness and  ill-posedness]
{Charge-Free Well-Posedness and Charge-Driven Norm Inflation 
for the Euler--Maxwell System}

\author{Haroune Houamed}
\address{Ko\c{c} University, Istanbul, Turkey}
\email{hhouamed@ku.edu.tr; haroune.houamed@nyu.edu}


\begin{document}

\begin{abstract}
We study the two-dimensional incompressible Euler--Maxwell system under the normal geometry in which the velocity and electric field are planar and the magnetic field is normal to the plane. We identify a structural distinction between the projected, charge-free formulation and the unprojected system. 

For the projected system, we establish a local well-posedness theory for initial velocities in $H^1(\mathbb R^2)$ with bounded initial vorticity and electromagnetic data in $H^{\nicefrac 32}(\mathbb R^2)$, under a strict sub-luminal condition on the velocity. The main ingredient is an endpoint time-like trace estimate for half-waves, which controls the magnetic gradient in $L_t^2$ along every fluid trajectory. This closes the Yudovich vorticity estimate without requiring an Eulerian Lipschitz bound on the electromagnetic field. 

For the unprojected system, the longitudinal electric mode produces an additional vorticity source term containing the charge density ``$\diver  E$''. In this case, we construct smooth, compactly supported radial electromagnetic data converging to zero in $H^{\nicefrac32}(\mathbb R^2)$, with zero initial velocity, whose unique global smooth solution exhibits a   vorticity norm inflation in $L^\infty_x$. Thus,  at the same electromagnetic Sobolev regularity, the Gauss constraint separates endpoint regularity propagation from charge-driven vorticity norm inflation.
\end{abstract}

\maketitle

\setcounter{tocdepth}{1}
\tableofcontents

\section{Introduction}

\subsection*{Purpose} 
The incompressible Euler--Maxwell system describes the interaction of an ideal conducting fluid with an electromagnetic field \cite{Buskmap93,Davidson01}. Even in two space dimensions, this interaction places two very different dynamics in  direct competition. The fluid vorticity is transported by a velocity field (of Yudovich regularity, say), whereas the electromagnetic variables propagate at the finite speed of light, and are damped through Ohm's law. In the absence of the electromagnetic field, a bounded vorticity is propagated by the classical Yudovich theory \cite{Yudovich}. However, once the Lorentz force is present, the vorticity is no longer conserved, and the central question becomes whether the Maxwell evolution supplies enough spacetime control to bound the vorticity's source terms.

This paper identifies a structural distinction in that question. At the same electromagnetic regularity $H^{\nicefrac32}({\mathbb R}^2)$, we prove an endpoint well-posedness for the charge-free projected formulation, and vorticity norm inflation for the charged unprojected formulation (see the next section for the precise setup). These conclusions are not two competing statements about a single Cauchy problem. They, however, occur on opposite sides of the Gauss constraint: the charge-free projection eliminates the longitudinal electric mode, while the charged system retains a non-propagating mode that can transfer a highly concentrated charge directly into the fluid vorticity. Thus, the distinction between the two theories is not one of Sobolev versus Besov summability, it is, more precisely, encoded in the formulation of Ohm's law and in the constraint imposed on the electric field. To the best of our knowledge, this is the first work that precisely identifies a rigorous distinction in the theory of the two models, below.

\subsection*{The two formulations}
We work on ${\mathbb R}^2$, although the fluid and electromagnetic fields take values
in ${\mathbb R}^3$.  Let $c>0$ denote the speed of light and $\sigma>0$ the electrical
conductivity.  The charge-free projected system is given by 
\begin{equation}\label{eq:EM-intro} \tag{EM-$\mathbb P$}
\left|~
\begin{aligned}
\text{\tiny(Euler equation)} &&&\partial_tu+u\cdot\nabla u +\nabla p=j\times B,
 \qquad  &\diver u=0 &,\\
 \text{\tiny(Amp\`ere  equation)}&&&\tfrac1c\partial_tE-\nabla\times B=-j,
 \qquad  &\diver E=0 &,\\
\text{\tiny(Faraday equation)}&&& \tfrac1c\partial_tB+\nabla\times E=0,
 \qquad  &\diver B=0 &,\\
 \text{\tiny(solenoidal Ohm law)} & & &j=\sigma\bigl(cE+\mathbb P(u\times B)\bigr),
\qquad  &\diver j=0 &,
\end{aligned}
\right.
\end{equation}
where $\mathbb P=\operatorname{Id}-\nabla\Delta^{-1}\diver$ is the Leray projector on divergence-free vector fields, while the charged unprojected system is instead formulated as
\begin{equation}\label{eq:unprojected-system}\tag{EM}
\left| ~
\begin{aligned}
 \text{\tiny(Euler equation)} &&&\partial_tu+u\cdot\nabla u+\nabla p=j\times B,
 \qquad &\diver u=0&,\\
\text{\tiny(Amp\`ere  equation)}&&& \tfrac1c\partial_tE-\nabla\times B=-j,\\
\text{\tiny(Faraday equation)}&&& \tfrac1c\partial_tB+\nabla\times E=0,
 \qquad &\diver B=0&,\\
\text{\tiny(compressible Ohm law)} &&& j=\sigma(cE+u\times B).
\end{aligned}
\right.
\end{equation}

Introducing the charge density $\rho\bydef \diver E$,  and taking the divergence of Amp\`ere's equation gives the corresponding continuity law
\begin{equation*} 
 \tfrac1c\partial_t\rho+\diver j=0.
\end{equation*}
Consequently, setting $\rho=0$ is dynamically compatible precisely when the current is divergence-free, i.e., the projection in \eqref{eq:EM-intro} enforces this compatibility. 

Throughout the paper, we use the word ``\emph{charged}'' as a shorthand for allowing $\diver E\ne0$ in \eqref{eq:unprojected-system}. We also impose the normal two-dimensional geometry
\begin{equation}\label{eq:normal-intro}
 u=(u_1,u_2,0),\qquad E=(E_1,E_2,0),\qquad B=(0,0,b),
\end{equation}
which is a structure that is propagated by both systems as soon as it is assumed to hold at the initial time $t=0$.  Under this geometry condition, it is readily seen that 
\begin{equation}\label{eq:vorticity-dichotomy}
 \curl(j\times b e_3)
 =-\diver(bj)
 =-j\cdot\nabla b-b\,\diver j.
\end{equation}
For \eqref{eq:EM-intro}, the last term vanishes and the vorticity
$$\omega \bydef \partial_1u_2-\partial_2u_1$$
 is then governed by the forced transport equation 
\begin{equation}\label{eq:vorticity-projected}
 (\partial_t+u\cdot\nabla)\omega=-j\cdot\nabla b.
\end{equation}
For \eqref{eq:unprojected-system}, however, the additional term $-b\,\diver j$ remains.  It is to be emphasized later on that formula \eqref{eq:vorticity-dichotomy} is the algebraic source of the well-posedness/ill-posedness dichotomy established, below, for \eqref{eq:EM-intro} and \eqref{eq:unprojected-system}, respectively.

These two models can be obtained from the corresponding viscous Navier--Stokes--Maxwell system through the vanishing viscosity limit, while the latter  had been rigorously derived in \cite{ASR19} from Vlasov--Maxwell--Boltzmann system via hydrodynamic limits. 

\section{Statements and discussion of the main results} We now state the main theorems of the paper, and provide a concise discussion about their relationship and proof strategies. The first statement is about the local well-posedness of the charge-free Euler--Maxwell system \eqref{eq:EM-intro}.

\begin{theorem}[Charge-free EM: local well-posedness]\label{thm:main}
Let  $c,\sigma>0$ be fixed, and consider    initial data enjoying the normal structure
\eqref{eq:normal-intro}, with $\diver u_0=\diver E_0=0,$ such that 
\begin{equation*} 
 u_0\in H^1({\mathbb R}^2), \qquad
 \omega_0=\operatorname{curl}_2u_0\in L^\infty({\mathbb R}^2),
\end{equation*}
\begin{equation*} 
 (E_0,B_0)\in H^{\frac{3}{2}}(\mathbb R^2) .
\end{equation*}
We assume further that there is a strict sub-luminal gap in the sense that 
\begin{equation*} 
 \norm{u_0}{L^\infty}\leq c-\delta_0,
 \qquad\text{for some }\delta_0>0.
\end{equation*}
Then, there exists $T=T\left(c,\sigma,\delta_0,
 \norm{u_0}{H^1},\norm{\omega_0}{L^\infty},
 \norm{(E_0,B_0)}{H^{3/2}}\right)>0$
and a unique solution of \eqref{eq:EM-intro} on $[0,T]$ such that
\begin{equation*}
	u\in C([0,T];H^1(\mathbb R^2)), 
	\qquad
 \omega\in L^\infty([0,T];L^2\cap L^\infty(\mathbb R^2)),
\end{equation*}
\begin{equation*}
	(E,B)\in C([0,T]; H^{\frac32}(\mathbb R^2)).
\end{equation*}
Moreover, the velocity field remains sub-luminal, i.e., 
\begin{equation}\label{eq:gap-preserved}
 \sup_{0\leq t\leq T}\norm{u(t)}{L^\infty}
 \leq c-\frac{\delta_0}{2},
\end{equation}
the energy inequality
\begin{equation*} 
 \norm{(u,E,B)(t)}{L^2(\mathbb R^2)}^2
 +\frac2\sigma\int_0^t\norm{j(s)}{L^2(\mathbb R^2)}^2d  s
 \leq \norm{(u_0,E_0,B_0)}{L^2(\mathbb R^2)}^2
\end{equation*}
holds, and, for every absolutely continuous curve $X:[0,T]\to\mathbb R^2$ with
$\|\dot X\|_{L^\infty}\leq c-\delta_0/2$ one has that 
\begin{equation}\label{eq:trace-main}
 \norm{\nabla b(t,X(t))}{L^2(0,T)}
 \lesssim _c \delta_0^{-\frac12} 
  \norm{(E_0,B_0)}{H^{\frac32}(\mathbb R^2)}
   .
\end{equation}
\end{theorem}
 
 \begin{remark}
 	It is to be emphasized later on in the proof of Theorem \ref{thm:main} that the data-to-solution map is continuous at zero-initial-data in the sense that 
 	\begin{equation*}
 		\norm {(u,\omega,E,B)}{L^\infty ([0,T];\dot H^1\times L^\infty\times H^\frac{3}{2}\times H^\frac{3}{2})} \lesssim \norm {(u_0,\omega_0,E_0,B_0)}{ \dot H^1\times L^\infty\times H^\frac{3}{2}\times H^\frac{3}{2} }.
 	\end{equation*}
 	
 \end{remark}

Our second main result concerns the charged system
\eqref{eq:unprojected-system}. In contrast with Theorem \ref{thm:main} for \eqref{eq:EM-intro}, we now construct a sequence of smooth, compactly
supported, normally structured data that generate a unique global smooth solution which exhibits a norm inflation phenomenon for the vorticity in the $L^\infty_x$ space, while the reference regularity for the electromagnetic field remains at the $\nicefrac 32$-Sobolev space. The precise statement of this is the following.

\begin{theorem}[Charged EM: norm inflation]\label{thm:main:2}
Let  $c,\sigma>0$ be fixed.  There are constants $C_0,c_*>0$ and, for every sufficiently
large integer $N$, smooth compactly supported normally structured initial data enjoying the normal structure \eqref{eq:normal-intro}
\begin{equation*}
	u_{0,N}=0,\qquad E_{0,N}:\mathbb R^2\to\mathbb R^2,
 \qquad B_{0,N}=b_{0,N}e_3,
\end{equation*}
such that
\begin{equation*} 
 \norm{E_{0,N}}{H^{3/2}(\mathbb R^2)}
 +\norm{B_{0,N}}{H^{3/2}(\mathbb R^2)}
 \leq C_0N^{-\frac 18},
\end{equation*}
while the unique corresponding solution of \eqref{eq:unprojected-system} is global and smooth, but at the observation time $t_N=N^{-\nicefrac18}$
 the vorticity   associated to this solution enjoys the bound 
\begin{equation}\label{eq:inflation}
 \norm{\omega_N(t_N)}{L^\infty(\mathbb R^2)}
 \geq c_*N^{\frac 18}.
\end{equation}
\end{theorem}

\begin{remark}
It is to be emphasized later on that  the electric data are curl free but not divergence-free; more precisely,
\begin{equation}\label{eq:charge-size}
 \diver E_{0,N}(0)=2N^{\frac 38 }.
\end{equation}
Moreover,   up to the observation time $t_N$,  the   velocity  fields remain small in   $L^\infty_x$, i.e., 
\begin{equation}\label{eq:u-small}
 \sup_{0\leq t\leq t_N}\norm{u_N (t)}{L^\infty}
 \leq C_0N^{-\frac78}.
\end{equation} 
\end{remark}

\begin{remark}
	It is to be emphasized later on that we are able to prove a stronger statement for the ill-posedness of \eqref{eq:unprojected-system}. More precisely, the initial electromagnetic field that we will construct below can actually be made small in the $H^s$ topology, for any $s<2$, with a vanishing $H^s$-norm in the asymptotic limit $N\to \infty$, while the $L^\infty_x$-inflation of the vorticity still remains valid.  In fact, this could further be improved by including a logarithmic correction in the construction of the initial electromagnetic field to recover a similar instability statement for electromagnetic field in the endpoint space $H^2$, as well. We choose  not to include the discussion of this refinement in the proof below, which goes beyond the scope of this article. Nevertheless,  we believe that, in this sense, the $H^2$ topology could   be sharp. This is based on the fact that the ill-posedness mechanism we illustrate later on is driven by the $L^\infty_x$ inflation of $\diver E_N$ that appears in the source term of the vorticity, which in turn originates from a choice of an initial data that exhibits the failure of the two-dimensional  Sobolev space in  $\dot H^s$ in $ \dot W^{1,\infty} $ for $s\leq 2$.  
\end{remark}

\subsection*{Background and the endpoint obstruction}

The global Yudovich theory for the normally structured projected system \eqref{eq:EM-intro} was
developed by Ars\'enio and the author in \cite{ArsenioHouamedAPDE}.  The argument from that work
combines the energy dissipation of the current, a normal structure product
law, and damped Strichartz estimates for Maxwell's equations.  At the
endpoint, the electromagnetic data are controlled in a Besov space with
$\nicefrac74$ derivatives.  This index arises from the two-dimensional wave estimate
used to place a full spatial derivative of the magnetic field in an
Eulerian spacetime norm of the form $L^2_tL^\infty_x$.  There, the  $\ell^1$ Besov
summation is essential at the relevant wave endpoint, for the corresponding
homogeneous Sobolev estimate being false \cite{FangWang}. 

It is important to also stress that the normal structure imposed in \cite{ArsenioHouamedAPDE} seems to be necessary for the well-posedness theory of the Euler--Maxwell system in the Yudovich class of solutions. We refer to \cite{HouamedIllposedness} for a discussion about this point. Besides this, the particular formulation of the Euler--Maxwell system in the projected form \eqref{eq:EM-intro} was also studied in the context of the non-relativistic limit ($c\to \infty$, in strong topologies) in two and three dimensions \cite{ArsenioHouamedIMRN,DJJ26}. 

The unprojected Ohm law in \eqref{eq:unprojected-system}, on the other hand, is also standard in
the viscous Navier--Stokes--Maxwell literature.  Global two-dimensional
regular solutions, local large-data theories, global small-data theories,
and large-energy results under additional electromagnetic control are also known, see for instance \cite{A19,ArsenioGallagher,AHB25,GermainIbrahimMasmoudi,SS11,Masmoudi}.  These
results underscore that a charged electric component is not an
instability mechanism in every used functional setting. More precisely, one can summarize this point into the statement that viscosity and stronger
solution spaces can stabilize the coupling, whence, loosely speaking,  one can expect that identical results could be established both models \eqref{eq:EM-intro} and \eqref{eq:unprojected-system}.  

To the best of our knowledge, the opposite  statements obtained for \eqref{eq:EM-intro} and \eqref{eq:unprojected-system} through Theorems \ref{thm:main} and \ref{thm:main:2}, respectively, formulate the first type of results identifying a crucial distinction between these two models in the specific inviscid Yudovich theory and at the electromagnetic Sobolev regularity $\nicefrac 32$. 

Once again, it is clear that the current paper does not discuss nor settle the optimality question of the regularity exponent $\nicefrac74$ from the global theory developed in \cite{ArsenioHouamedAPDE} for \eqref{eq:EM-intro}, which reflects the cost of controlling the magnetic
Lipschitz norm simultaneously at every spatial point. 

\subsection*{Strategy of proofs: the crux in the distinct findings of the paper}

The  key observation behind the well-posedness result stated by Theorem \ref{thm:main} for the charge-free system \eqref{eq:EM-intro} stems from the fact that the vorticity
is governed by the forced transport equation \eqref{eq:vorticity-projected}, and hence it
only samples $\nabla b$ along fluid trajectories. This  suggests
that the Eulerian norm is stronger than what the nonlinear problem actually
requires, which is merely a control for $\nabla b$ along the characteristics of the velocity field. The first part of the paper, i.e., the proof of Theorem \ref{thm:main}, exactly exploits that observation by establishing the key linear estimate 
\begin{equation}\label{eq:timelike-trace-intro}
 \norm{\nabla e^{\pm ict|D|}f(X(t))}{L^2_t(I)}
 \lesssim \delta^{-1/2}\norm{f}{\dot H^{3/2}},
\end{equation}
for every flow map $X\in W^{1,\infty}(I;\mathbb R^2)$ satisfying the bound 
\begin{equation*}
 \|\dot X\|_{L^\infty(I)}\leq c-\delta,
 \qquad\text{for some }\delta>0.
\end{equation*}

For a fixed curve, this follows from polar Fourier coordinates.  Indeed, the
phase in a direction $\theta\in\mathbb S^1$ is given by
\begin{equation*}
	\phi_\theta^\pm(t)=\pm ct+\theta\cdot X(t), 
\end{equation*}
which, under the Lipschitz control of the flow, satisfies that 
\begin{equation*}
	|(\phi_\theta^\pm)'(t)|\geq \delta.
\end{equation*}
Given this, a one-dimensional change of variables followed by Plancherel's theorem yields
\eqref{eq:timelike-trace-intro}.  Note in passing that, when $X$ is constant, the bound \eqref{eq:timelike-trace-intro} is closely related to --- and in fact generalizes --- the reversed $L^\infty_xL^2_t$ estimate established by Fang and Wang in 
\cite{FangWang}. 

  The proof of Theorem \ref{thm:main} then implements  that idea at the lower
regularity $H^{\nicefrac32}$.  It builds on the normal product estimates and damped
Maxwell analysis of \cite{ArsenioHouamedAPDE,ArsenioHouamedIMRN}, as well as
the compactness and stability methods developed in
\cite{ArsenioHouamedIMRN,ArsenioHassainiaHouamed} to construct the  unique local solution of \eqref{eq:EM-intro}. At this stage, we tend to believe that the $H^{\nicefrac 32}$ is optimal for the local theory. This claim will not be discussed here, but it is based on the sharpness of the $\nicefrac 32$ regularity threshold required in the reversed Strichartz estimate \eqref{eq:timelike-trace-intro}. We refer again to \cite{FangWang} for the latter statement.

We turn now to discuss the mechanism behind the ill-posedness of the charged \eqref{eq:unprojected-system} system at the same Sobolev regularity $\nicefrac 32$ for the electromagnetic field, i.e., Theorem \ref{thm:main:2}. As briefly pointed out earlier, in the case of \eqref{eq:unprojected-system}, the source of   norm inflation of the vorticity is the presence of the charge density $\diver E$ in vorticity forcing, which generated from expanding Ohm's law in \eqref{eq:vorticity-dichotomy}. To precisely exploit the impact of this term, we consider the simplest possible configuration corresponding to looking for radial solutions to \eqref{eq:unprojected-system} of the form  
\begin{equation*}
	 u=V(t,r)e_\theta,\qquad
 E=R(t,r)e_r+S(t,r)e_\theta,\qquad
 B=b(t,r)e_3,
\end{equation*}
where $r=|x|$, and the pair $(e_r,e_\theta)$ denotes the radial basis in $\mathbb R^2$. The vorticity in this setting is given by 
\begin{equation}\label{vorticity:intro}
\omega(t,r)=\frac1r\partial_r\bigl(rV(t,r)\bigr)
 =\partial_rV(t,r)+\frac{V(t,r)}r.
\end{equation} 
 Besides filtering out the transport term from the vorticity equation, the essential benefit of this ansatz is the reduction of \eqref{eq:unprojected-system} system to the equivalent radial reformulation 
\begin{equation}\label{eq:radial-system}
\left |~
\begin{aligned}
 \partial_tV&=-\sigma b(cR+bV),\\
 \partial_tR&=-\sigma c(cR+bV),\\
 \partial_tS&=-c\,\partial_rb-\sigma c^2S,\\
 \partial_tb&=-c\left(\partial_rS+\frac Sr\right).
\end{aligned}
\right.
\end{equation}
 
 The main observation here is that the transverse pair $(S,b)$ solves a closed linear Maxwell system.  Once $b$ is known, the remaining variables $(V,R)$ evolve according to a pointwise linear ordinary differential equation. 
 
 To exhibit the norm inflation phenomenon, we choose an initial data of the form 
 \begin{equation*}
 	b_0\equiv \beta , \qquad  S_0  \equiv 0 \quad \text{ on a ball } B_{R_*} , \text{ for some constants } R_*>0 \text{ and } \beta\in \mathbb R, 
 \end{equation*}
 while we take $V_0\equiv 0$ and $R_0$ to be the radial component of an electric field $E_0$ exhibiting the failure of the Sobolev embedding of $\dot H^2(\mathbb R^2)$ in $\dot W^{1,\infty}(\mathbb R^2)$. More precisely, $E_0$  is defined as a highly concentrated (rescaled) potential of a smooth function $\phi \in C_c^\infty (\mathbb R^2)$ satisfying 
 \begin{equation*} 
 \Phi(x)=\frac{|x|^2}{2}\quad\text{on }B_1\qquad \text{ and}
 \qquad
 \supp\Phi\subset B_2.
\end{equation*}
 
  To put all of these together, we further exploit the finite speed of  propagation of Maxwell's equation to keep $b$ exactly constant on the support of the concentrated electric profile (for a short time), which eventually allows us to explicitly solve the ODE governing the evolution of $V$, and show that the variable $R$ feeds the  norm inflation at the origin ($x=0$) of the vorticity given in terms of $V$ by   formula  \eqref{vorticity:intro}.

\subsection*{Organization of the paper} The lineup of the next sections of the  paper splits as follows. First, in Section \ref{sec:preliminaries},    we briefly discuss the normal structure product law estimate used throughout the proofs, which originally builds on the analysis laid out   in \cite{ArsenioHouamedAPDE}. Then, in Section \ref{sec:trace}, we establish the central time-like reversed Strichartz-type estimate for half-wave semi-groups,   which will be used in the well-posedness theory of \eqref{eq:EM-intro}. There, we also discuss how to actually apply this estimate to the Maxwell system. After that, Section \ref{sec:apriori} contains the a priori estimates for \eqref{eq:EM-intro} which are employed thereafter in Section \ref{sec:existence} to prove Theorem \ref{thm:main}. 

In Section \ref{sec:ill-posedness}, we collect all the necessary building blocks for the instability analysis of the charged Euler--Maxwell system  \eqref{eq:unprojected-system}. The findings from that part of the paper will be utilized thereafter in Section \ref{sec:proof:thm2} to design the initial data leading to the norm inflation phenomenon, and eventually the conclusion of proof of Theorem \ref{thm:main:2}.

For completeness, we supplement our analysis by a summarizing endnote, and an appendix containing a brief discussion of the sharpness of the $\nicefrac 32$-Sobolev regularity required in the reversed Strichartz-type trace estimate.

\section{Preliminaries}\label{sec:preliminaries}

This section supplies some basic definitions and briefly recalls some important results that will serve in the subsequent sections. 

\subsection*{Notations and function spaces}First, we  agree to use the standard symbol ``$\lesssim $'' for inequalities involving absolute constants that do not depend on the core variables of the problem. Specifically, this type of constants, which are also allowed to differ from one line to another, will sometimes be kept and denoted by ``$C$''. 

For a given positive parameter $r>0$, and a position $x_0$, we denote the ball centered at $x_0$ with radius $r$ by $\mathcal B_r(x_0)$, with  the convention $\mathcal B_r(0)=\mathcal B_r$. We also use this same notation in both one and two-dimensions. 

 For a planar vector field $F=(F_1,F_2,0)$ and a scalar $g$, we denote
\begin{equation*}
 \operatorname{curl}_2F=\partial_1F_2-\partial_2F_1,
 \qquad
 \mathcal Dg=(\partial_2g,-\partial_1g,0)
\end{equation*}
which translates the three dimensional rotation as  
$$\nabla\times F=(\operatorname{curl}_2F)e_3 \quad \text{ and } \quad 
\nabla\times(ge_3)=\mathcal Dg.$$  
We also use the following notation  
\begin{equation*}
	F^\perp = (-F_2,F_1, 0)
\end{equation*}
for  the vector representing the orthogonal of $F$.

Throughout the paper, we use the standard (in-)homogeneous Sobolev space, which can be defined via the Littlewood--Paley dyadic blocks.  For
$s\in\mathbb R$, these spaces are endowed by the (semi-)norms 
\begin{equation*}
 \norm{f}{H^s(\mathbb R^2)}^2\bydef \sum_{q\ge-1}2^{2qs}\norm{\Delta_qf}{L^2(\mathbb R^2)}^2,
 \qquad
 \norm{f}{\dot H^s(\mathbb R^2)}^2\bydef \sum_{q\in \mathbb Z}2^{2qs}\norm{\dot \Delta_qf}{L^2(\mathbb R^2)}^2.
\end{equation*} 
We refer to \cite{BahouriCheminDanchin} for the precise definitions of the dyadic blocks $\Delta_q$, and for a complete discussion about the standard paraproduct and embedding results used throughout the paper. 

\subsection*{Biot--Savart law and the flow map} In the context of our interest in this paper, the  regularity of the vorticity $\omega$  is limited, which mostly belongs to $\dot H^{-1}(\mathbb R^2)\cap L^2(\mathbb R^2) \cap L^\infty (\mathbb R^2)$. Its relation with the velocity field $u$ is given by the Biot--Savart law associated to the two-dimensional equation 
\begin{equation}\label{eq:BS}
	u=\nabla^{\perp} \Delta^{-1}\omega.
\end{equation}
The Biot--Savart law and the Calder\'on--Zygmund multiplier theorem yield
(see \cite[Section 7.1.1]{BahouriCheminDanchin} and
\cite[Section 6.2.3]{Grafakos})  
\begin{equation*}
	\norm {u}{\dot W^{1,p}(\mathbb R^2)}\lesssim  \frac{p^2}{p-1} \norm \omega{L^p(\mathbb R^2)},
\end{equation*}
for all $p\in (1,\infty)$, where $\dot W^{1,p}(\mathbb R^2)$ denotes the space of functions with exactly one derivative being $L^p$-integrable. This estimate is sharp in the sense that a bounded vorticity does not generate a Lipschitz velocity field, in general.
Yet, another useful information that one can obtain is the   boundedness of the velocity field, i.e.,
\begin{equation}\label{eq:velocity-interpolation}
 \norm{u}{L^\infty(\mathbb R^2)}
 \lesssim \norm{u}{L^2(\mathbb R^2)}^{\frac12}\norm{\omega}{L^\infty(\mathbb R^2)}^{\frac12},
\end{equation} 
whose proof can be done via a standard two-dimensional interpolation argument.

 Beyond this, and through the elliptic relation \eqref{eq:BS}, a bounded vorticity, however, gives rise to a log-Lipschitz velocity field $u\in LL(\mathbb R^2)$. It turns out that this is enough to uniquely define the characteristics $X_a$, for every given position $a\in \mathbb R^2$, as the solution of the integral equation 
\begin{equation*}
	X_a(t)=a+\int_0^t u(s,X_a(s))d s,
\end{equation*}
for any $t\in [0,T]$ as soon as $u \in L^\infty( [0,T]; LL(\mathbb R^2))$. An immediate consequence is that 
\begin{equation}\label{flow:time:lip}
	|X_a(t)-X_a(s)|
 \le\|u\|_{L^\infty_{t,x}}|t-s|,
\end{equation}
for every $t,s\in [0,T]$, yielding thus to the assertions  
$$X_a\in W^{1,\infty}(0,T;\mathbb R^2) \quad \text { and } \quad\dot  X_a(t)=u(t,X_a(t)),$$
 for almost every $t\in [0,T]$.
 
 \subsection*{The normal-structure product law} Roughly speaking, the essential feature of the normal structure \eqref{eq:normal-intro} is apparent in the identity 
 \begin{equation*}
 \mathbb P(u\times be_3)=-\mathbb P(bu^\perp)=-[\mathbb P,b]u^\perp.
\end{equation*}
  The advantage of this structure has been deeply exploited in \cite{ArsenioHouamedAPDE}, for the first time in the analysis of the Euler--Maxwell systems, to establish a global well-posedness result for the charge-free model \eqref{eq:EM-intro}. In particular, it allows us to extend the classical range of regularity parameters in the paraproducts of vectors obeying that structure. 
 
 In the next lemma, we only state the unique product estimate that utilizes the specific structure of vectors in \eqref{eq:normal-intro}, which   will be used in our proofs, later on. This corresponds to a particular case covered by estimate (3-13) from \cite[Lemma 3.4]{ArsenioHouamedAPDE}.

\begin{lemma}[Endpoint normal product law]\label{lem:product} 
For any  smooth  divergence-free vector field $u$ and a scalar function $b:\mathbb R^2\to\mathbb R$, it holds that 
\begin{equation*}
 \norm{\mathbb P(u\times be_3)}{\dot H^{\frac32}(\mathbb R^2)}
 \lesssim \norm{u}{L^\infty\cap \dot H^1(\mathbb R^2)} 
 \norm{b}{\dot  H^{\frac32}(\mathbb R^2)}.
\end{equation*}
The same estimate holds for time-dependent fields pointwise in time.
\end{lemma}

 \section{The time-like endpoint estimate}\label{sec:trace}
 
In this section, we establish the central key estimate to control   the gradient of the magnetic field along the characteristics of the velocity, which appears to be the main source term in the vorticity equation for the charge-free \eqref{eq:EM-intro} system.
 
First, we prove a time-like trace estimate for the half-wave semi-group, which we extend thereafter to the undamped Maxwell propagator.   

\begin{proposition}[Time-like half-wave trace]\label{prop:trace}
Let $I\subset\mathbb R$ be an interval, and consider a flow map $X\in W^{1,\infty}(I;\mathbb R^2)$ being such that 
\begin{equation}\label{eq:curve-gap}
 \|\dot X\|_{L^\infty(I)}\leq  c-\delta,
\end{equation}
for some $\delta>0$.
Let $m(D)$ be a scalar- or matrix-valued  zero-order Fourier multiplier whose
symbol is bounded.  Then, for either sign, it holds that
\begin{equation}\label{eq:timelike-trace}
 \norm{\nabla m(D)e^{\pm ict|D|}f(X(t))}{L^2_t(I)}
 \lesssim \delta^{-1/2}\norm{m}{L^\infty_\xi}
 \norm{f}{\dot H^{\frac32}(\mathbb R^2)},
\end{equation}
for any Schwartz function $f $. In the matrix-valued case, the symbol norm above is the essential supremum of the operator norm.
\end{proposition}

\begin{remark}
	Initially, the trace is the pointwise trace of a Schwartz function whose Fourier transform vanishes near the origin.  We realize $\dot H^{\nicefrac32}$ as the completion of that class, modulo constants.  The estimate above defines the trace for a general element of $\dot H^{\nicefrac32}$, and the gradient makes it independent of the constant representative.
\end{remark}

\begin{remark}[Relation with  reversed Strichartz estimate]\label{rem:fang-wang} Time-space reversed Strichartz-type estimates are known and not new in the literature on dispersive propagators. In particular,  
 Fang and Wang established in \cite[Proposition 4]{FangWang}   the bound  
\begin{equation}\label{FangWang:es}
 \norm{e^{it|D|}f(x)}{L^\infty_x((\mathbb R^d);L^q_t(\mathbb R))}
 \lesssim\norm{f}{\dot H^{\frac d2-\frac 1q}(\mathbb R^d)},\qquad 2\leq q<\infty.
\end{equation} 
For $d=q=2$, applying one derivative gives the constant-curve case of
\eqref{eq:timelike-trace}. This stationary estimate does not by itself imply
the moving estimate that we need in this work, i.e.,   generally speaking:
\begin{equation*}
	 \sup_x\|F(\cdot\,,x)\|_{L^2_t}
 \quad\hbox{does not control}\quad
 \|F(t,X(t))\|_{L^2_t}.
\end{equation*} 
The time-like trace estimate we establish in Proposition \ref{prop:trace} above can thus be interpreted as a generalization of \eqref{FangWang:es}.
\end{remark}

\begin{proof}
  In polar Fourier coordinates
$\xi=r\theta$, $r>0$ and $\theta\in\mathbb S^1$, one has that 
\begin{equation*}
 \nabla m(D)e^{\pm ict|D|}f(X(t))
 =C\int_{\mathbb S^1}\int_0^\infty
 e^{ir\phi_\theta(t)}\,ir^2\theta\,m(r\theta)
 \widehat f(r\theta)d  rd \theta,
\end{equation*}
where the phase $\phi_\theta $ is given by
\begin{equation*}
	\phi_\theta(t)\bydef \pm ct+X(t)\cdot\theta, \quad \text{ for all } t\in I.
\end{equation*}
Note in passing that, due to the assumption on the Lipschitz bound of the flow map, it holds   
\begin{equation*}
	\phi_\theta'(t)\geq   \delta >0 ,
\end{equation*}
for the plus sign, and 
\begin{equation*}
	\phi_\theta'(t)\le-\delta < 0,
\end{equation*}
for the minus sign.  All in all, in either case, the phase function $\phi_\theta$ is one-to-one on
$I$, and its inverse is $\delta^{-1}$-Lipschitz.  Thus, introducing the function
\begin{equation*}
	s\mapsto  g_\theta(s)=\int_0^\infty e^{irs}\,ir^2\theta\,m(r\theta)
 \widehat f(r\theta)d  r,
\end{equation*}
we find, by virtue of a one-dimensional change-of-variables and Plancherel theorem, that
\begin{equation}\label{eq:theta-trace}
 \norm{g_\theta\circ\phi_\theta}{L^2(I)}^2
 \leq \delta^{-1}\norm{g_\theta}{L^2(\mathbb R)}^2
 \leq C\delta^{-1}\norm{m}{L^\infty}^2
 \int_0^\infty r^4|\widehat f(r\theta)|^2d  r.
\end{equation}
Therefore, Minkowski's inequality in $\theta$, followed by Cauchy--Schwarz on $\mathbb S^1$,   yields that 
\begin{align*}
 \norm{\nabla m(D)e^{\pm ict|D|}f(X(t))}{L^2_t(I)}
 &\leq C\delta^{-1/2}\norm{m}{L^\infty}
 \left(\int_{\mathbb S^1}\int_0^\infty
 r^4|\widehat f(r\theta)|^2d  rd \theta\right)^{\frac12}\\
 &=C\delta^{-1/2}\norm{m}{L^\infty}\norm{f}{\dot H^{\frac32}},
\end{align*}
because $d \xi=rd  rd \theta$.  Completion in the homogeneous norm gives
the asserted trace.  In particular, for inhomogeneous data this agrees with
the extension obtained from the density of $\mathcal S$ in $H^{3/2}$.
\end{proof}

A direct consequence of the previous proposition is the estimate for the inhomogeneous half-wave problem, which we  discuss next.
\begin{corollary}[Retarded trace estimate]\label{cor:duhamel-trace}
Under the assumptions of Proposition \ref{prop:trace}, if
$I=[0,T]$ and $F\in L^1(0,T;\dot H^{\nicefrac32})$, then one has that 
\begin{equation}\label{eq:duhamel-trace}
 \left\|
 \left.\nabla\int_0^t m(D)e^{\pm ic(t-s)|D|}F(s)d  s
 \right|_{x=X(t)}
 \right\|_{L^2_t(0,T)}
 \leq C\delta^{-1/2}\norm{m}{L^\infty}
 \norm{F}{L^1_t\dot H^{\frac32}}.
\end{equation}
\end{corollary}

\begin{proof}
At first, we suppose that $F$ is smooth.  By Minkowski's inequality in $s$,
the retarded time triangle gives
\begin{align*}
 &\left\|\int_0^t \nabla m(D)e^{\pm ic(t-s)|D|}F(s)(X(t))d  s
 \right\|_{L^2_t(0,T)}\\
 &\qquad\leq \int_0^T
 \left\|\mathbf 1_{[s,T]}(t)\nabla m(D)e^{\pm ic(t-s)|D|}
 F(s)(X(t))\right\|_{L^2_t(0,T)}d  s.
\end{align*}
Now, for fixed $s$, the norm on the right-hand side is the norm on $[s,T]$ of
\begin{equation*}
	t\mapsto \nabla m(D)e^{\pm ict|D|}
 \bigl(e^{\mp ics|D|}F(s)\bigr)(X(t)).
\end{equation*}
Proposition \ref{prop:trace} then applies  to the restriction of $X$ to $[s,T]$. Hence, in view of the fact that the half-wave group is unitary on $\dot H^{\nicefrac32}$, and integration in $s$
leads to \eqref{eq:duhamel-trace}.  Once again, an approximation argument in
$L^1_t\dot H^{\nicefrac32}$ yields the result for general functions in that space.
\end{proof}

\subsection*{Application to the normal Maxwell system}  
Let $D=|D|$ be the Fourier multiplier of symbol $|\xi|$, and set
\begin{equation*}
 q=D^{-1}\operatorname{curl}_2E,
 \qquad
 \ell=D^{-1}\operatorname{curl}_2j.
\end{equation*}
 Since $E$ and $j$ are
divergence-free, $E=-\nabla^\perp D^{-1}q$, and, similarly, $j$ is recovered from
$\ell$ by a zero-order Riesz transform.  This reconstruction is precisely
where the Gauss constraint is used. In terms of these new variables, the   Maxwell equations from
\eqref{eq:EM-intro} can be recast as
\begin{equation*}
 \partial_tq=cDb-c\ell,
 \qquad
 \partial_tb=-cDq.
\end{equation*}
Thus, the transformed variables $w_ \pm=q\pm ib$   solve the forced half-wave
\begin{equation}\label{eq:halfwaves}
 (\partial_t  \pm  icD)w_\pm=- c\ell.
\end{equation}

The next corollary is a direct consequence of Proposition \ref{prop:trace} and Corollary \ref{cor:duhamel-trace}.

\begin{corollary}[Maxwell trace]\label{cor:maxwell-trace}
Let $(E,b)$ be the mild solution of the normal Maxwell equations on $[0,T]$,
with $\diver E=\diver j=0$, $(E_0,b_0)\in H^{\nicefrac32}$, and
$j\in L^1(0,T;H^{\nicefrac32})$.  For any curve satisfying
\eqref{eq:curve-gap}, it then holds that 
\begin{equation*}
 \norm{\nabla b(t,X(t))}{L^2_t(0,T)}
 \lesssim \delta^{-\frac12}
 \left(\norm{(E_0,b_0)}{\dot H^{3/2}}
 +c\norm{j}{L^1_t\dot H^{3/2}}\right).
\end{equation*} 
\end{corollary}

\begin{proof}
Writing  \eqref{eq:halfwaves} in  Duhamel formula, and applying
Proposition \ref{prop:trace} and Corollary \ref{cor:duhamel-trace}, yields that 
\begin{equation*}
	\norm{\nabla w_{\pm}(t,X(t))}{L^2_t(0,T)}
 \lesssim \delta^{-\frac12}
 \left(\norm{(E_0,b_0)}{\dot H^{\frac32}}
 +c\norm{j}{L^1_t\dot H^{\frac32}}\right).
\end{equation*}
 From the definition of $w_\pm $, i.e., writing $2ib= w_+-w_-$, this estimate yields the  one for $b$.
\end{proof}

\section{A priori estimates for the charge-free system}\label{sec:apriori}

All calculations in this section are first made for smooth solutions.  They
will be applied uniformly to the approximation introduced in Section \ref{sec:existence}, below. For this section, we adopt the following notations:  
\begin{equation*}
 A_0\bydef \norm{(u_0,E_0,b_0)}{L^2},\qquad
 M_0\bydef \norm{(E_0,b_0)}{H^{\frac32}},
\end{equation*}
and, for every $0<T'<T$, we further denote (for $r\in \{2,\infty\} $)
\begin{equation*}
	\Omega_{r,T'}\bydef \norm{\omega}{L^\infty(0,T';L^r)} , \qquad U_{T'} \bydef \norm{u}{L^\infty(0,T';H^1\cap L^\infty)},
\end{equation*}
\begin{equation*}
	M_{T'}\bydef \norm{(E,b)}{L^\infty(0,T';H^{\frac32})}, \qquad J_{T'}\bydef \norm{j}{L^\infty(0,T';H^{\frac32})}, 
\end{equation*}
and 
\begin{equation*} 
 G_{T'}(\delta) \bydef 
 \sup_{\norm{\dot X}{L^\infty(0,T')}\leq c-\delta}
 \norm{\nabla b(t,X(t))}{L^2(0,T')}.  
\end{equation*} 

\subsection*{Energy, Maxwell regularity, and the current}

The first set of estimates (namely the basic energy bounds \eqref{eq:energy-identity} and \eqref{eq:maxwell-X}, below) in this section are standard by now, and have been used in previous works such as \cite{ArsenioGallagher,ArsenioHouamedIMRN,ArsenioHouamedAPDE}. However, for       completeness of our presentation, we provide concise  details, below.

Taking the $L^2$ scalar product of the fluid equation with $u$, and of the
Maxwell equations with $(E,b)$, gives the identity 
\begin{equation*}
	\frac12\frac{d }{d  t}\norm{(u,E,b)}{L^2}^2
 =-\int_{\mathbb R^2}j\cdot(u\times B+cE)d  x.
\end{equation*}
Since $j$ is divergence-free, it is orthogonal to gradients, whence 
\begin{equation*}
	\int j\cdot(u\times B)=\int j\cdot\mathbb P(u\times B).
\end{equation*}
Ohm's law now yields the exact identity
\begin{equation}\label{eq:energy-identity}
 \frac12\frac{d }{d  t}\norm{(u,E,b)}{L^2}^2
 +\frac1\sigma\norm{j}{L^2}^2=0.
\end{equation}

We turn now to discuss Sobolev regularity of the electromagnetic fields. To that end, we apply $\Delta_q$ to Maxwell's equations and take the dyadic energy to find that 
\begin{equation}\label{eq:dyadic-maxwell}
 \frac12\frac{d }{d  t}
 \left(\norm{\Delta_kE}{L^2}^2+\norm{\Delta_kb}{L^2}^2\right)
 +\sigma c^2\norm{\Delta_kE}{L^2}^2
 =-\sigma c \left\langle\Delta_kE, \Delta_k \mathbb P(u\times B)\right\rangle .
\end{equation}
Taking the appropriate $\ell^2$ norm in $k$, and using
Lemma \ref{lem:product} to control the source term, yields, for some absolute constant $C>0$, that 
\begin{equation}\label{eq:maxwell-X}
 M_{T'}\leq M_0+C\sigma c \int_0^{T'} U_{\tau }M_{\tau} d\tau ,
\end{equation}
whereby, Gr\"onwall lemma implies that
\begin{equation}\label{eq:maxwell-gronwall}
 M_{T'}\leq M_0\exp(C\sigma cT'U_{T'}).
\end{equation}
In view of the definition of the solenoidal  Ohm's law in \eqref{eq:EM-intro}, applying Lemma \ref{lem:product} once again together with the two-dimensional embedding
$H^{3/2}\hookrightarrow L^\infty(\mathbb R^2)$ lead to  the control 
\begin{equation}\label{eq:current-X}
 \norm{j}{L^\infty((0,T')\times\mathbb R^2)}\lesssim J_{T'} 
 \leq C\sigma(c+U_{T'})M_{T'}.
\end{equation}

\subsection*{Vorticity along the flow}

We recall that, for vector fields obeying the normal structure \eqref{eq:normal-intro}, it holds that 
\begin{equation*}
 \operatorname{curl}_2(j\times be_3)
 =-\diver(bj)=-j\cdot\nabla b,
\end{equation*}
because $\diver j=0$.     Let $X(t,a)$ be
the volume-preserving Yudovich flow of $u$ and, for a fixed position $a\in \mathbb R^2$, we use the shorthand notation
$X_a(t)=X(t,a)$.  

 We want to apply Proposition \ref{prop:trace} with this flow map, which requires that   the single trajectory $t\mapsto X_a(t)$ is Lipschitz (in time).
This requirement is available in the present Yudovich class, which we previously discussed in Section \ref{sec:preliminaries}. More precisely, the time-Lipschitz estimate \eqref{flow:time:lip} of the flow map $X_a$ yields, under the condition 
$$\norm{u}{L^\infty_{t,x}}\leq c-\delta,$$
 that 
$$\|\dot X_a\|_{L^\infty_t}\leq c-\delta.$$   
Therefore, as soon as the above $L^\infty_{t,x}$ control of the velocity is at hand, Corollary \ref{cor:maxwell-trace} applies to deduce that 
\begin{equation}\label{eq:G-bound}
 G_{T'}(\delta)
 \leq C\delta^{-\frac12}\left(M_0+cT'J_{T'}\right).
\end{equation}
On the other hand, along a characteristic, the vorticity equation can be recast as
\begin{equation*}
	 \omega(t,X_a(t))=\omega_0(a)
 -\int_0^t j(s,X_a(s))\cdot\nabla b(s,X_a(s))d  s.
\end{equation*}
Thus, Cauchy--Schwarz inequality in time and \eqref{eq:G-bound} yield that 
\begin{equation}\label{eq:omega-infty}
 \Omega_{\infty,T'}
 \le\norm{\omega_0}{L^\infty}
 +T'^{\frac12}\norm{j}{L^\infty_{t,x}}G_{T'}(\delta).
\end{equation}
Besides, the $L^2$ transport estimate, together with the trivial estimate
$$\norm{\nabla b}{L^2_x}\leq M_{T'},$$
 also gives that
\begin{equation}\label{eq:omega-two}
 \Omega_{2,T'}
 \le\norm{\omega_0}{L^2}+CT'\norm{j}{L^\infty_{t,x}}M_{T'}.
\end{equation}
Finally, by virtue of interpolation inequality \eqref{eq:velocity-interpolation}, the previous two estimates of the vorticity lead to the control for the velocity field
\begin{equation}\label{eq:U-bound}
 U_{T'}
 \leq C\left(A_0+\Omega_{2,T'}
 +A_0^{\frac12}\Omega_{\infty,T'}^{\frac12}\right).
\end{equation}

\subsection*{Persistence of the light-speed gap}
Next, we discuss the time-persistence of initial light-speed gap estimate. To that end,  after applying Leray's projector $\mathbb P$ to the velocity equation, we obtain by a direct $L^2_x$ estimate  that 
\begin{equation}\label{eq:ut-bound}
 \norm{\partial_tu}{L^2_x}
 \leq C\left(\norm{u}{L^\infty_x}\norm{\omega}{L^2_x}
 +\norm{j}{L^\infty_x}\norm{b}{L^2_x}\right).
\end{equation}
Now, employing the interpolation inequality \eqref{eq:velocity-interpolation} this time to $u(t)-u_0$ yields that
\begin{equation}\label{eq:gap-continuity}
 \norm{u(t)-u_0}{L^\infty}
 \leq C\norm{u(t)-u_0}{L^2_x}^{\frac12}
 \left(\Omega_{\infty,T'}+\norm{\omega_0}{L^\infty_x}\right)^{\frac12},
\end{equation}
where we also utilized \eqref{eq:omega-infty}.
Note in passing that, assuming all norms used here are bounded, the right-hand side in this last inequality tends to zero uniformly as $T'\downarrow0$ by
virtue of \eqref{eq:ut-bound}.

We are now in a position to close the bootstrap argument for the local a priori estimates, which we summarize in the next statement.
\begin{proposition}[Uniform short-time bound]\label{prop:bootstrap}
For data satisfying the hypotheses of Theorem \ref{thm:main}, there is a
time $T>0$, depending only on the quantities displayed in that theorem, such
that every smooth solution satisfying the energy identity obeys, on $[0,T]$,
\begin{equation}\label{eq:bootstrap-conclusion}
\begin{gathered}
 M_T\le2M_0,\qquad
 \Omega_{2,T}\le\norm{\omega_0}{L^2}+1,\qquad
 \Omega_{\infty,T}\le\norm{\omega_0}{L^\infty}+1,\\
 \norm{u}{L^\infty((0,T)\times\mathbb R^2)}
 \leq c-\frac{\delta_0}{2}.
\end{gathered}
\end{equation} 
\end{proposition}

\begin{proof}
We bootstrap the assertions 
\begin{equation}\label{eq:bootstrap-assumptions}
 M_{T'}\le2M_0,\quad
 \Omega_{2,T'}\le\norm{\omega_0}{L^2}+1,\quad
 \Omega_{\infty,T'}\le\norm{\omega_0}{L^\infty}+1,\quad
 \norm{u}{L^\infty_{t,x}}\leq c-\frac{\delta_0}{2},
\end{equation}
for $0< T' <T$.
To that end, setting
\begin{equation*}
	K\bydef 1+c+A_0+\norm{\omega_0}{L^2}
 +A_0^{\frac12}\bigl(1+\norm{\omega_0}{L^\infty}\bigr)^{\frac12}
\end{equation*}
yields, in view of  \eqref{eq:U-bound} and \eqref{eq:current-X}, that 
$$U_{T'}\leq CK$$ 
and
\begin{equation*}
	J_{T'}+\norm{j}{L^\infty_{t,x}}
 \leq J_*\bydef C\sigma(c+K)M_0.
\end{equation*}
Moreover, \eqref{eq:G-bound} evaluated at $\delta=\nicefrac{ \delta_0}2$ yields that 
\begin{equation*}
	 G_{T'}(\tfrac{\delta_0}{2})
 \leq G_*\bydef  C\delta_0^{-\frac12}(M_0+cJ_*),
\end{equation*}
as long as $T'\le1$.  Finally, \eqref{eq:ut-bound} is bounded by
\begin{equation*}
	Q_*\bydef \left(c(1+\norm{\omega_0}{L^2})+J_*A_0\right).
\end{equation*}

Now, we choose $T\le1$ so small that
\begin{equation*}
 C\sigma cKT\le\log(\tfrac32),\qquad
 CTJ_*M_0\le\tfrac12,\qquad
 CT^{\frac12}J_*G_*\le\tfrac12,
\end{equation*}
and
\begin{equation}\label{eq:T-gap-choice}
 C(TQ_*)^{\frac12}
 \bigl(1+2\norm{\omega_0}{L^\infty}\bigr)^{\frac12}
 \le\frac{\delta_0}{4}.
\end{equation}
Accordingly, estimates \eqref{eq:maxwell-gronwall}, \eqref{eq:omega-two}, and \eqref{eq:omega-infty} improve the first three bounds in
\eqref{eq:bootstrap-assumptions}. Moreover, estimates \eqref{eq:ut-bound},
\eqref{eq:gap-continuity}, and \eqref{eq:T-gap-choice} give
\begin{equation*}
	\norm{u(t)}{L^\infty}
 \leq c-\delta_0+\frac{\delta_0}{4}
 <c-\frac{\delta_0}{2},
\end{equation*}
which improves the fourth bound in \eqref{eq:bootstrap-assumptions}. All in all, continuity in $T'$ closes the bootstrap argument, ensuring the validity of \eqref{eq:bootstrap-conclusion} up to some time $T>0$.  This completes the proof.
\end{proof}

\section{Existence by a structure-preserving approximation}
\label{sec:existence}
In this section, we discuss the details of the approximation used to construct the solution claimed by Theorem \ref{thm:main}, based on the a priori estimates established in Proposition \ref{prop:bootstrap}. Although the overall analysis laid out here has become standard by now, we believe that it is important to make sure we provide a complete justification of our a priori estimates. More specifically, the approximation used in
\cite[Section~3.2]{ArsenioHouamedAPDE} adds a vanishing viscosity to
the velocity equation. Although this regularization is convenient for
compactness, it destroys the exact transport structure of the approximate
vorticity equation, which is essential to the trajectory-based argument
used here.

\begin{proof}[Proof of Theorem \ref{thm:main}]
	
We begin by fixing a nonnegative, radial, even function $\rho\in C_c^\infty(\mathbb R^2)$, with a unite-integral.  We then introduce, for $n\in \mathbb N^*$, the mollifier  setting 
$$\rho_n(\cdot)=n^2\rho(n\cdot)\quad  \text{ and } \quad 
\cJ_nf=\rho_n*f.$$
  Thus, $\cJ_n$ is self-adjoint, commutes with derivatives
and $\mathbb P$, and is uniformly bounded on all spaces used before. Moreover, positivity also gives that 
\begin{equation*}
	 \|\cJ_nu_0\|_{L^\infty}\le\|u_0\|_{L^\infty},\qquad
 \|\cJ_n\omega_0\|_{L^\infty}\le\|\omega_0\|_{L^\infty},
\end{equation*}
and convolution contracts the remaining initial norms.  Therefore, the
approximating data have the same gap $\delta_0$, and the common time furnished by Proposition \ref{prop:bootstrap} is independent of $n$.

For the mollified initial data $\cJ_n(u_0,E_0,b_0)$, we consider the approximate system
\begin{equation}\label{eq:approx-system}
\left | ~
\begin{aligned}
 \partial_tu_n+u_n\cdot\nabla u_n+\nabla p_n
 &=\cJ_n(j_n\times B_n),
 &\diver u_n&=0,\\
 \partial_tE_n-c\mathcal Db_n&=-cj_n,
 &\diver E_n&=0,\\
 \partial_tb_n+c\operatorname{curl}_2E_n&=0,\\
 j_n&=\sigma\left(cE_n+
 \mathbb P((\cJ_nu_n)\times B_n)\right),
 &\diver j_n&=0,
\end{aligned}
\right.
\end{equation}
where we denote $B_n=b_ne_3$. Given the mollified source terms, the existence of a unique global solution of this system, for each fixed $n\in \mathbb N^*$, can be done in a routine way. Indeed, the Maxwell $H^m$ energy is bounded on finite time intervals, for every $m\geq 0$, because
\begin{equation*}
	\norm{\cJ_nu_n}{W^{m,\infty}}
 \leq C_{m,n}\norm{u_n}{L^2},
\end{equation*}
and the energy (see 
\eqref{eq:approx-energy}, below) controls the right-hand side, above.  Also, it happens that
\begin{equation*}
 \norm{\curl\cJ_n(j_n\times B_n)}{L^\infty}
 \leq C_n\norm{j_n\times B_n}{L^1}
 \leq C_n\norm{j_n}{L^2}\norm{B_n}{L^2},
\end{equation*}
whose time integral is finite on each bounded interval by the energy.  The usual two-dimensional vorticity continuation criterion and high-order energy estimates then prevent finite-time loss of smoothness.   

\subsection*{Exact cancellation and uniform vorticity control}

Self-adjointness gives that 
\begin{align*}
 \int u_n\cdot\cJ_n(j_n\times B_n)d  x
 &=\int \cJ_nu_n\cdot(j_n\times B_n)d  x\\
 &=-\int j_n\cdot((\cJ_nu_n)\times B_n)d  x\\
 &=-\int j_n\cdot\mathbb P((\cJ_nu_n)\times B_n)d  x.
\end{align*}
Consequently, \eqref{eq:approx-system} satisfies the exact energy identity
\begin{equation}\label{eq:approx-energy}
 \frac12\frac{d }{d  t}\norm{(u_n,E_n,b_n)}{L^2}^2
 +\frac1\sigma\norm{j_n}{L^2}^2=0.
\end{equation}
Furthermore, the approximate vorticity $\omega_n \bydef \curl_2 u_n$ is governed by the transport equation 
\begin{equation*}
 (\partial_t+u_n\cdot\nabla)\omega_n
 =-\cJ_n(j_n\cdot\nabla b_n).
\end{equation*}
Now, let $X_n(t,a)$ be the volume-preserving flow of $u_n$.  Since
\begin{equation*}
 \cJ_n(j_n\cdot\nabla b_n)(X_n(t,a))
 =\int\rho_n(y)j_n(t,X_n(t,a)-y)
 \cdot\nabla b_n(t,X_n(t,a)-y)d  y,
\end{equation*}
and every curve $t\mapsto X_n(t,a)-y$ has the same speed as $X_n(t,a)$, it is readily seen that the proof of \eqref{eq:omega-infty}, followed by integration in
$y$, applies without change and gives that
\begin{equation*}
 \norm{\omega_n}{L^\infty_TL^\infty}
 \le\norm{\cJ_n\omega_0}{L^\infty}
 +CT^{1/2}\norm{j_n}{L^\infty_{t,x}}
 G_{n,T}(\delta),
\end{equation*}
which is the analogue of \eqref{eq:omega-infty} for the approximate solution of \eqref{eq:approx-system}. 

Additionally, the $L^2$ estimate also follows because $\cJ_n$ is bounded on $L^2$.
Moreover, $\cJ_nu_n$ may replace $u_n$ in Lemma \ref{lem:product}, since
\begin{equation*}
	\norm{\cJ_nu_n}{L^\infty\cap \dot H^1(\mathbb R^2)}
 \le\norm{u_n}{L^\infty\cap \dot H^1(\mathbb R^2)}.
\end{equation*}
Accordingly, we deduce that all estimates from Section \ref{sec:apriori} are thus uniform in $n$.

\subsection*{Compactness and passage to the limit}

Let $T$ be supplied by Proposition \ref{prop:bootstrap}. The previous estimates provide us with the assertions (uniformly in $n$) that
\begin{equation}\label{eq:uniform-bounds}
\begin{aligned}
 &u_n\text{ is bounded in }L^\infty(0,T;H^1),
 &&\omega_n\text{ is bounded in }L^\infty(0,T;L^2\cap L^\infty),\\
 &(E_n,b_n)\text{ is bounded in }L^\infty(0,T;H^{\nicefrac32}), \quad 
 &&j_n\text{ is bounded in }L^\infty(0,T;H^{\nicefrac32}) .
\end{aligned}
\end{equation}
The equations further ensure that
\begin{equation*}
 \partial_tu_n\text{ is bounded in }L^\infty(0,T;L^2),
 \qquad
 \partial_t(E_n,b_n)\text{ is bounded in }L^\infty(0,T;H^{\nicefrac12}),
\end{equation*}
uniformly in $n$, which follow by employing the bounds in \eqref{eq:uniform-bounds} to control the nonlinear terms from \eqref{eq:approx-system}. 

Hence, Aubin--Lions--Simon \cite{Simon1986} compactness lemma provides a
subsequence such that
\begin{equation*}
 (u_n,E_n,b_n)\longrightarrow(u,E,b)
 \quad\text{strongly in }C([0,T];L^2_{\mathrm{loc}}).
\end{equation*}
The same lemma also gives strong convergence in lower local Sobolev spaces. Finally, Fatou's property and weak lower semicontinuity give all bounds on the solution as stated in Theorem \ref{thm:main}.  This said, one can show in a routine way that $(u,E,b)$ solves \eqref{eq:EM-intro} in the sense of distributions.

In addition to that, the weak limit of $\omega_n$ belongs to
$L^\infty_t(L^2\cap L^\infty)$ and, since
\begin{equation*}
	j\cdot\nabla b\in L^\infty(0,T;L^2),
\end{equation*}
and $u$ is a Yudovich velocity, the transport equation
\eqref{eq:vorticity-projected} is renormalized.  Its measure-preserving Yudovich flow
satisfies (a.e. in $t,x$)
\begin{equation*} 
 \omega(t,X(t,x))=\omega_0(x)
 -\int_0^t(j\cdot\nabla b)(s,X(s,x))d  s.
\end{equation*} 
Since the source belongs to $L^1(0,T;L^2)$, measure preservation and
Minkowski's inequality give  that $\omega\in C([0,T];L^2)$. 
At last, the mild Maxwell formula and Lemma \ref{lem:product} show that its
source belongs to $L^1(0,T;H^{\nicefrac32})$.  Thus, by strong continuity of the free Maxwell group on $H^{3/2}$,  we deduce that 
$(E,b)\in C([0,T];H^{\nicefrac32})$.  The gap
\eqref{eq:gap-preserved} and trace estimate \eqref{eq:trace-main} follow by
the already established bounds and Corollary \ref{cor:maxwell-trace}.  
  This completes the existence part of Theorem \ref{thm:main}.

As for the uniqueness part, we emphasize that the uniqueness statement already proved in \cite{ArsenioHouamedAPDE} (see more precisely the proof of Theorem 3.2 therein in page 1370) covers the function spaces involved in Theorem \ref{thm:main}. This concludes the proof of the theorem.
\end{proof}

\section{The building blocks for the charged system}\label{sec:ill-posedness}
This section is devoted to   introducing the required tools and results to establish  the norm inflation phenomenon for the charged system \eqref{eq:unprojected-system}.  First, we recall the standard notation for the cylindrical basis in three-dimensions
\begin{equation*}
	e_r\bydef (\cos\theta,\sin\theta,0),
 \qquad
 e_\theta\bydef (-\sin\theta,\cos\theta,0).
\end{equation*}
In this basis, we seek radial solutions to \eqref{eq:unprojected-system} of the form
\begin{equation}\label{eq:radial-ansatz}
 u(t,x)=V(t,r)e_\theta,
 \qquad
 E(t,x)=R(t,r)e_r+S(t,r)e_\theta,
 \qquad
 B(t,x)=b(t,r)e_3,
\end{equation}
where $r=|x|$.  Note in passing that such velocity and magnetic field are automatically divergence-free, and the vorticity in this case is given by
\begin{equation}\label{eq:radial-vorticity}
\omega(t,r)=\frac1r\partial_r\bigl(rV(t,r)\bigr)
 =\partial_rV(t,r)+\frac{V(t,r)}r.
\end{equation}
 For this ansatz, the next lemma equivalently reduces the Euler--Maxwell system \eqref{eq:unprojected-system} to a simpler, partially coupled, system of equations.

\begin{lemma}[Radial equations]\label{lem:radial-equations}
Smooth vector fields of the form \eqref{eq:radial-ansatz} solve
\eqref{eq:unprojected-system} if and only if
\begin{equation}\label{eq:radial-system}
\left |~
\begin{aligned}
 \partial_tV&=-\sigma b(cR+bV),\\
 \partial_tR&=-\sigma c(cR+bV),\\
 \partial_tS&=-c\,\partial_rb-\sigma c^2S,\\
 \partial_tb&=-c\left(\partial_rS+\frac Sr\right), 
\end{aligned}
\right.\tag{Rad-eq}
\end{equation}
and the pressure is recovered, up to a function of time, through
\begin{equation}\label{eq:pressure}
 \partial_rp=\frac{V^2}{r}+\sigma cbS.
\end{equation}
\end{lemma}

\begin{proof}
First, we recall the standard identities for the three-dimensional cylindrical basis
\begin{equation*}
	e_r\times e_3=-e_\theta,
 \qquad e_\theta\times e_3=e_r,
\end{equation*}
which we constantly employ hereafter without recalling. To start, 
 observing that $ u\times B=bV e_r$, we infer from Ohm's   law in \eqref{eq:unprojected-system} that
\begin{equation}\label{eq:radial-current}
 j=j_re_r+j_\theta e_\theta,
 \qquad
 j_r=\sigma(cR+bV),
 \qquad
 j_\theta=\sigma cS.
\end{equation}
Therefore, it follows that
\begin{equation*}
	 j\times B=bj_\theta e_r-bj_re_\theta.
\end{equation*}
Furthermore, noting that the advection term reduces to
\begin{equation*}
	 u\cdot\nabla u=-\frac{V^2}{r}e_r,
\end{equation*}
the tangential and radial components of Euler's equation thus produce the equations
\begin{equation*}
	\partial_tV=-bj_r,
 \qquad
 -\frac{V^2}{r}=-\partial_rp+bj_\theta.
\end{equation*}
Using \eqref{eq:radial-current} yields the first equation in
\eqref{eq:radial-system} and \eqref{eq:pressure} altogether.

On the other hand, recalling, for a radial scalar $b$, that
\begin{equation*}
	\nabla\times(be_3)=-\partial_rb\,e_\theta,
\end{equation*}
the radial and tangential components of Amp\`ere's equation thus produce in turn the equations
\begin{equation*}
	\frac1c\partial_tR=-j_r,
 \qquad
 \frac1c\partial_tS+\partial_rb=-j_\theta,
\end{equation*}
which give the second and third equations in \eqref{eq:radial-system}.
Finally, writing that
\begin{equation*}
	\nabla\times E
 =\left(\partial_rS+\frac Sr\right)e_3,
\end{equation*}
  Faraday's equation yields the fourth equation in \eqref{eq:radial-system}. This concludes the proof of the lemma.
\end{proof}

The crucial feature of \eqref{eq:radial-system} is its triangular structure.
The transverse variables $(S,b)$ solve a closed linear system, independently of
$(V,R)$.  Once $b$ has been determined, $(V,R)$ are recovered, at each radius, by the
linear system
\begin{equation}\label{eq:longitudinal-matrix}
 \partial_t
 \begin{pmatrix}V\\  R\end{pmatrix}
 =-\sigma
 \begin{pmatrix}
  b^2&bc\\bc&c^2
 \end{pmatrix}
 \begin{pmatrix}V\\ R\end{pmatrix}.
\end{equation}

Let us next discuss the pathway of constructing a smooth solution of \eqref{eq:radial-system}-\eqref{eq:pressure} that will serve the purpose of exhibiting a norm inflation phenomenon for \eqref{eq:unprojected-system}. We first begin with a basic consequence of the finite speed of propagation for the transverse variables $(S,b)$. To that end, it is simpler to introduce the corresponding constant-coefficient damped Maxwell system 
\begin{equation}\label{eq:TR:rad}
	\left| ~
	\begin{aligned}
		\partial_tE^{\mathrm{tr}}&=c\nabla\times B-\sigma c^2E^{\mathrm{tr}},
		\\
		\partial_t B&=-c\nabla\times E^{\mathrm{tr}}
		\\
		 (E^{\mathrm{tr}}, B)&= (Se_\theta, be_3)
	\end{aligned}
	\right.\tag{Max-tr}
\end{equation} 
and discuss the finite speed of propagation property for the latter reformulation.

\begin{lemma}\label{lemma:SB:finite-speed}
	Suppose that the radial profiles $S_0,b_0$ generate smooth compactly supported Cartesian fields of the form \eqref{eq:radial-ansatz}.  Then \eqref{eq:TR:rad} has a global smooth solution that remains compactly supported on every finite time interval.
\end{lemma}

\begin{proof}
	The standard Fourier representation, or equivalently the energy method for a symmetric hyperbolic system with a zeroth-order damping term, produces a unique global smooth solution from smooth data.  Moreover, finite speed of  propagation also entails that the support of this solution remains compact on every bounded time interval, as soon as this is assumed to hold initially. Indeed, and more precisely, this follows from the fact that, if the smooth initial data is identically zero on  some ball $ \mathcal B_{r_0}(x_0)$, then  $( E^{\mathrm{tr}}, B)(t,x) =0$ 
whenever 
\begin{equation*}
	|x-x_0| + ct <r_0,
\end{equation*}
 To see this, we first write the local energy identity  
 \begin{equation*}
 	 \partial_t \mathcal E  + c \nabla \cdot (E^{\mathrm{tr}} \times B) = -\sigma c^2|E^{\mathrm{tr}}|^2, \qquad \mathcal E (t,x) \bydef \frac{1}{2} |( E^{\mathrm{tr}}, B)(t,x)|^2,
 \end{equation*}
 for all $(t,x)\in \mathbb R\times \mathbb R^2$. Then, introducing, for all $t\in [0,r_0/c)$, the ball 
 \begin{equation*}
 	\Omega_t \bydef \mathcal B_{\rho (t)}(x_0), \qquad \rho(t)\bydef r_0 -ct,
 \end{equation*}
 and using Reynolds transport formula (which can be proved directly in the case of the shrinking ball $\Omega_t$ through polar coordinates) 
 \begin{equation*}
 	\frac{d}{dt} \int_{\Omega_t} \mathcal E (t,x) dx = \int_{\Omega_t}  \partial_t\mathcal E (t,x) dx -c  \int_{\partial \Omega_t} \mathcal E (t,\sigma) d\sigma ,
 \end{equation*}
 we obtain, after integrating the local energy identity over the ball $\Omega_t$ and using the divergence formula, that 
 \begin{equation*}
 	\frac{d}{dt} \int_{\Omega_t} \mathcal E (t,x) dx= - c  \int_{\partial \Omega_t} \mathcal E (t,\sigma) d\sigma - c  \int_{\partial \Omega_t}  (E^{\mathrm{tr}} \times B)\cdot \vec n(t,\sigma) d\sigma -\sigma c^2 \int_{\Omega_t} |E^{\mathrm{tr}} (t,x)|^2 dx.
 \end{equation*}
Now, observe that the boundary flux satisfies the pointwise  bound 
\begin{equation*}
	| (E^{\mathrm{tr}} \times B)\cdot \vec n |  \leq | E^{\mathrm{tr}}| |B| \leq \mathcal E,
\end{equation*}
 whence
 \begin{equation*}
 	- c  \int_{\partial \Omega_t} \mathcal E (t,\sigma) d\sigma - c  \int_{\partial \Omega_t}  (E^{\mathrm{tr}} \times B)\cdot \vec n(t,\sigma) d\sigma\leq 0.
 \end{equation*}
 All in all, we deduce that 
 \begin{equation*}
 	\int_{\Omega_t} \mathcal E (t,x) dx \leq \int_{\mathcal B _{r_0}} \mathcal E (0,x) dx=0,
 \end{equation*}
 for all $t\in [0,r_0/c)$. This establishes the desired finite speed of propagation property, which eventually yields the persistence of the compact-support property of the initial data by a straightforward argument. More precisely, utilizing the argument above, one can show that if the initial data is supported inside some compact set $K\subset \mathbb R^2$, then for every time $t>0$, it holds that 
 \begin{equation*}
 	\supp (E^{\mathrm{tr}}(t,\cdot), B(t,\cdot)) \subset K + \overline {\mathcal B}_{ct}.
 \end{equation*}
 This completes the proof of the lemma.
\end{proof}

We now build on our previous findings for the transverse variables, and provide a full statement on the global dynamics of the complete system of equations \eqref{eq:radial-system}.

\begin{lemma}[Global smooth radial dynamics]\label{lem:global-radial} Suppose that the radial profiles $V_0,R_0,S_0,b_0$ generate smooth compactly supported Cartesian fields of the form \eqref{eq:radial-ansatz}.  Then \eqref{eq:radial-system} has a global smooth solution, and \eqref{eq:pressure} defines a smooth pressure.  If $V_0=R_0=0$ at some  fixed radius $r$, then $V(t,r)=R(t,r)=0$ for every $t\ge0$.
\end{lemma}

\begin{proof} 
We have already discussed the existence and all the desired properties of $(S,b)$ in Lemma \ref{lemma:SB:finite-speed}, above. Thus, we only focus our attention here on the analysis for $(V,R)$. With $b$ fixed, the first two equations of $(V,R)$ from \eqref{eq:radial-system} are equivalently recast in the matrix form given by  \eqref{eq:longitudinal-matrix}, which is a smooth linear ODE at each
$r$.  This ODE cannot blow up in finite time and, in fact, one has that
\begin{equation}\label{eq:pointwise-dissipation}
 \frac12\partial_t(V^2+R^2)
 =-\sigma(bV+cR)^2\le0.
\end{equation}
Differentiating the ODE with respect to $r$, and proceeding inductively, shows
that all radial derivatives remain finite on bounded time intervals.  Additionally, the
homogeneous character of the ODE yields the support assertion.

For completeness, we also point out that the compatibility conditions at the polar-coordinate
origin are also preserved.  Indeed, smooth radial Cartesian data have the form
\begin{equation*}
	V_0(r)=r\,v_0(r^2),\qquad
  R_0(r)=r\,q_0(r^2),\qquad
 S_0(r)=r\,s_0(r^2),\qquad
   b_0(r)=k_0(r^2)
\end{equation*}
near $r=0$, for some smooth one-variable functions $v_0,q_0,s_0,k_0$.  Thus, rotational invariance and smoothness of the Cartesian transverse Maxwell flow preserve the corresponding representation
\begin{equation*}
	S(t,r)=r\,s(t,r^2),\qquad b(t,r)=k(t,r^2).
\end{equation*}
The coefficient matrix in \eqref{eq:longitudinal-matrix} then depends smoothly
on $(t,r^2)$, so uniqueness for the pointwise ODE yields
\begin{equation*}
	 V(t,r)=r\,v(t,r^2),\qquad R(t,r)=r\,q(t,r^2).
\end{equation*}
As for the pressure, it follows that the right-hand side of \eqref{eq:pressure} is of the form $r\,h(t,r^2)$, for some smooth function $h$.  It therefore integrates to a smooth radial pressure at $r=0$.  Since $\partial_rp$ is compactly supported at each fixed time, the pressure is globally smooth.  This completes the construction of a global smooth solution of the full system.
\end{proof}

We  conclude this section by highlighting two special consequences of the finite-speed of propagation proved in the previous lemma for $(S,b)$, and apply it thereafter to produce an explicit solution to the system of ODE governing the variables $(V,R)$ in a particular case of initial data. These consequences are stated in the next two corollaries., which are specifically considered the building blocks used in the ill-posedness proof of Theorem \ref{thm:main:2}, later on.

\begin{corollary} [Magnetic plateau]\label{cor:plateau}
Assume, for some $\beta\in \mathbb R$ and  $r_*>0$, that 
\begin{equation*}
	b_0(r)=\beta\quad \text{ and } \quad  S_0(r)=0,\qquad  \text{ for all }r\in [0, r_*].
\end{equation*}
Then the solution of the transverse subsystem satisfies 
\begin{equation*}
 b(t,r)=\beta\quad \text{ and } \quad S(t,r)=0,
 \qquad \text{ for all } r+ct<r_*.
\end{equation*}
\end{corollary}

\begin{proof} 
The constant pair $(S,b)=(0,\beta)$ is a solution of the transverse system.
The difference between the given solution and this constant solution has
vanishing initial data in the ball $\mathcal B _{r_*}$.   Thus, the claimed conclusion follows from  the finite speed  principale proved for \eqref{eq:TR:rad} in Lemma \ref{lemma:SB:finite-speed}. 
\end{proof}

\begin{corollary} [Explicit longitudinal solution]\label{cor:explicit}
Suppose that $b(t,r)=\beta$ on a spacetime region $\Lambda_{c}$, containing the initial time, and that $V(0,r)=0$ there.
Then
\begin{equation*}
 V(t,r)
 =-\frac{\beta c}{c^2+\beta^2}
 \left(1-e^{-\sigma(c^2+\beta^2)t}\right)R_0(r)
\end{equation*}
at every point whose vertical time segment remains in that region.
\end{corollary}

\begin{proof}
Setting
\begin{equation*}
	q\bydef cR+\beta V,
\end{equation*}
we obtain, similarly to \eqref{eq:pointwise-dissipation}, that $q$ is governed by the equation
\begin{equation*}
	\partial_tq=-\sigma(c^2+\beta^2)q,
 \qquad q(0,r)=cR_0(r),
\end{equation*}
for all $(t,r)$ in the spacetime region $\Lambda_{c}$ where $b(t,r)=\beta$.
Thus, solving this ODE yields that 
\begin{equation*}
	q(t,r)=cR_0(r)e^{-\sigma(c^2+\beta^2)t}, \quad \text{for all } (t,r) \in \Lambda_{c}.
\end{equation*}
Since $\partial_tV=-\sigma\beta q$ and $V(0,r)=0$, integrating the expression of $q$ in time inside the region $\Lambda_{c}$  finally leads to the explicit formula for $V$.
\end{proof}

\section{Norm inflation for the charged system}\label{sec:proof:thm2}

This section is exclusively devoted to the proof of Theorem \ref{thm:main:2}.
\begin{proof}[Proof of Theorem \ref{thm:main:2}]

Let   $\Phi\in C_c^\infty(\mathbb R^2)$  be a radial function satisfying
\begin{equation*}
 \Phi(x)=\frac{|x|^2}{2}\quad\text{on }\mathcal  B_1\qquad \text{ and}
 \qquad
 \supp\Phi\subset \mathcal  B_2,
\end{equation*}
 and  let $F$ denote its gradient, i.e., 
\begin{equation*}
 F\bydef \nabla\Phi.
\end{equation*}
Then $F$ is a smooth, compactly supported radial vector field, i.e., a radial function pointing in the radial direction. Moreover, it is curl-free, and satisfies that
\begin{equation}\label{eq:F-affine}
 F(x)=x\quad \text{ and } \quad  \diver F(x)=2,\qquad \text{for all }x\in \mathcal B_1.
\end{equation}
We further fix another radial cutoff $\rho\in C_c^\infty(\mathbb R^2)$ such that
\begin{equation}\label{eq:rho}
 \rho=1\quad\text{on }\mathcal  B_4.
\end{equation}

We give a construction of   initial data that exhibits the failure of the embedding $\dot H^s(\mathbb R^2)$ in $\dot W^{1,\infty}(\mathbb R^2)$, for $s<2$, justifying along the way our claim in second remark after the statement of Theorem \ref{thm:main:2}. For the precise statement of that theorem, one can simply take $s=\nicefrac 32$, below. So, we fix $s\in [\nicefrac 32,2)$, and set  $\delta=\frac{2-s}{4}\in (0,\nicefrac 18]$. 
Accordingly, and for every $N\ge1$, we define 
\begin{equation}\label{eq:parameters}
 \alpha_N=N^{-(s-1+\delta)},
\qquad
\varepsilon_N =t_N=N^{-\delta},
\end{equation}
and set
\begin{equation}\label{eq:data}
 u_{0,N}=0,
 \qquad
 E_{0,N} =\alpha_NF(N\cdot),
 \qquad
 B_{0,N} =\varepsilon_N\rho(\cdot)e_3.
\end{equation}
A direct scaling argument implies that 
\begin{equation}\label{eq:E-H32}
 \norm{E_{0,N}}{H^s}
 \leq \alpha_NN^{s-1}\norm{F}{H^s}
 =\varepsilon_N\norm{F}{H^{s}},
\end{equation}
while
\begin{equation}\label{eq:B-H32}
 \norm{B_{0,N}}{H^{s}} 
 =\varepsilon_N\norm{\rho}{H^{s}},
\end{equation}
thereby yielding an initial electromagnetic fields whose $H^{s}$ norm is of size $O(\varepsilon_N)$.
Additionally, \eqref{eq:F-affine} entails that
\begin{equation}\label{eq:E-linear}
 (E_{0,N})|_{\mathcal B_{1/N}}=\alpha_NN \textnormal{Id}  
\end{equation}
and, consequently, that
\begin{equation}\label{eq:div-large}
 \diver E_{0,N}(0)=2\alpha_NN
 =2N^{2-s-\delta}=2N^{3\delta}.
\end{equation}

Now, observe that the initial data given by \eqref{eq:data} have the radial form
\eqref{eq:radial-ansatz}.  More precisely, in the notation of \eqref{eq:radial-ansatz}, the initial data satisfy that
\begin{equation*}
	V_N(0,r)=0,\qquad E_{0,N}(x)=R_N(0,r)e_r, \qquad S_N(0,r)=0.
\end{equation*}
  Lemma \ref{lem:global-radial} therefore produces a global smooth solution of \eqref{eq:radial-system} associated to this initial data. The small size of this initial data in the Sobolev space $H^{s}$, for all $s\in [\nicefrac32, 2)$, is ensured by \eqref{eq:E-H32}--\eqref{eq:B-H32}, while  \eqref{eq:div-large} is precisely \eqref{eq:charge-size} (in the particular case $s=\nicefrac 32$).

 Since $\rho=1$ on $B_4$ and $S_N(0,r)=0$ we deduce, by virtue of Corollary \ref{cor:plateau},  that 
\begin{equation}\label{eq:constant-on-support}
 b_N(t,r)=\varepsilon_N,\qquad S_N(t,r)=0,
\end{equation}
as long as  $r+ct<4$. 
Moreover, $\operatorname{supp}F\subset \mathcal B_2$ implies that $\operatorname{supp}R_N(0,\cdot)\subset\mathcal B_{\nicefrac 2N}$. Recalling, furthermore, that our observation time is defined as $t_N=N^{-\delta } \ll1$, for a sufficiently large $N$, makes the condition  $2/N+ct_N<4$ always satisfied.  Hence, it follows that \eqref{eq:constant-on-support} holds throughout
$[0,t_N]\times\supp R_N(0,\cdot)$.  The support assertion in
Lemma \ref{lem:global-radial} shows that no longitudinal component can appear
outside this set.  
Thus, applying Corollary \ref{cor:explicit} yields that 
\begin{equation}\label{eq:VN-formula}
 V_N(t,r)
 =-K_N(t)R_N(0,r),
\end{equation}
where
\begin{equation}\label{eq:KN}
 K_N(t)=
 \frac{\varepsilon_Nc}{c^2+\varepsilon_N^2}
 \left(1-e^{-\sigma(c^2+\varepsilon_N^2)t}\right).
\end{equation}
On the other hand, owing  to \eqref{eq:E-linear},
\begin{equation*}
	R_N(0,r)=\alpha_NNr,\quad \text{ for all }0\leq r<\frac1N.
\end{equation*}
Consequently, \eqref{eq:radial-vorticity} and
\eqref{eq:VN-formula} yield the exact identity
\begin{equation}\label{eq:omega-exact}
 \omega_N(t,0)=-2K_N(t)\alpha_NN.
\end{equation}

Notice next that, as soon as $0\leq z\le1$, one has that $1-e^{-z}\geq  z/2$, and, since
\begin{equation*}
	\sigma(c^2+\varepsilon_N^2)t_N\longrightarrow0, \text{ as } N\to \infty,
\end{equation*}
this inequality applies at $t=t_N$ for all sufficiently large $N$.  All in all, we deduce from 
\eqref{eq:KN} and \eqref{eq:omega-exact} that 
\begin{equation*}
	\begin{aligned}
		 |\omega_N(t_N,0)|
 &\ge
 2\frac{\varepsilon_Nc}{c^2+\varepsilon_N^2}
 \frac{\sigma(c^2+\varepsilon_N^2)t_N}{2}\alpha_NN\notag\\
 &=\sigma c\,\varepsilon_N\alpha_NNt_N.
	\end{aligned}
\end{equation*}
Substituting the values of the parameters from  \eqref{eq:parameters}, alongside the definition of $t_N=N^{-\delta}$, gives that 
\begin{equation*}
 |\omega_N(t_N,0)|
 \geq  \sigma c\,N^{\delta},
\end{equation*}
whereby \eqref{eq:inflation} holds for the constant $c_*=\sigma c$.

To conclude, it remains to show that the velocity itself stays small.  To that end, owing to the elementary inequality
$1-e^{-z}\leq z$, which is valid  for all $z\ge0$, \eqref{eq:KN} implies that 
\begin{equation*}
	0\leq K_N(t)\le\sigma\varepsilon_Nct,
 \quad \text{ for all } 0\leq t\leq t_N.
\end{equation*}
Employing this bound to control the velocity component in  \eqref{eq:VN-formula} yields that 
\begin{align*}
 \sup_{0\leq t\leq t_N}\norm{u_N(t)}{L^\infty}
 &\leq \sigma c\,\varepsilon_Nt_N\norm{R_N(0)}{L^\infty}\\
 &\lesssim \varepsilon_Nt_N\alpha_N = N^{\delta-1} \ll 1 ,  
\end{align*}
which yields the $L^\infty_x$ bound claimed in \eqref{eq:u-small} (where the particular case there corresponds to the value $\delta=\nicefrac 18$,  equivalently $s=\nicefrac 32$ in this proof). This completes the proof.
\end{proof}

\section*{Endnote}

 The mechanism underlying local endpoint well-posedness for \eqref{eq:EM-intro} can be summarized in one sentence: the normal-structure product
law propagates the $H^{\nicefrac 32}$ regularity of the Maxwell field, while strict sub-luminality converts this regularity into an $L^2_t$ trace estimate for the magnetic gradient along every fluid trajectory. This is precisely the control required in the vorticity equation to propagate the initial $L^\infty_x$ bound. The resulting argument establishes local well-posedness for charge-free projected data in $H^{\nicefrac 32}$, without requiring the electromagnetic field to be Lipschitz in the Eulerian variables.

For the charged, unprojected system \eqref{eq:unprojected-system}, by contrast, the established $L^\infty_x$ norm inflation of the vorticity is driven by the charge density $\diver E$. At the level of the endpoint estimates, this instability reflects the failure of the Sobolev embedding
\begin{equation*}
	  \dot H^s(\mathbb R^2)
    \not\hookrightarrow
    \dot W^{1,\infty}(\mathbb R^2), \quad s\leq 2.
\end{equation*}

Thus, as long as the longitudinal charge mode is excluded from the phase space, that is, as long as $\diver E=0$, the Euler--Maxwell system is locally well-posed  for vorticities in the Yudovich class and electromagnetic fields with at least $H^{\nicefrac 32}$ regularity. When the charge density is present, that is, when $\diver E\neq0$, it can instead generate an instantaneous instability of the vorticity in the Yudovich class, even for electromagnetic fields possessing substantial Sobolev regularity. These conclusions therefore lie on opposite sides of a structural boundary: they concern different phase spaces, rather than conflicting regularity statements for the same Cauchy problem.
\appendix

\section{Sharpness of the time-like half-wave trace} 
For completeness, we discuss in the next proposition the sharpness of the Sobolev space $\dot H^{\frac{3}{2}}(\mathbb R^2)$ that appears in the statement of Proposition \ref{prop:trace}. In particular, the exponent $\nicefrac 32$ in \eqref{eq:timelike-trace}  cannot be lowered,
even when the Fourier multiplier $m\equiv1$, the trajectory is the constant curve $X(t)=0$, and the time interval is fixed. This can further be generalized {\em mutatis mutandis} to a similar statement in any dimension $d\geq 1$ by substituting the Sobolev space above with  $\dot H^{\frac{d}{2}+\frac{1}{2}}(\mathbb R^d)$.

\begin{proposition}[Sharpness for uniform trajectory traces]\label{prop:sharp-trace} 
Let $I\subset\mathbb{R}$ be an interval with nonempty interior and let $c>0$ be fixed. For every $s<\nicefrac 32$, it holds that
\begin{equation*}
 \sup_{ \overset{f\in\mathcal{S}(\mathbb{R}^{2})}{\norm {f}{\dot H^s}\neq 0}}
 \frac{ \left\| \left.\nabla e^{\pm  ict|D|}f\right|_{x=0}
 \right\|_{L^{2}_{t}(I)}}{ \|f\|_{\dot H^{s}(\mathbb{R}^{2})}} =\infty.
\end{equation*} 
\end{proposition}

\begin{proof}
Fix $t_{0}\in\operatorname{int}(I)$, and choose $\eta>0$ such that $[t_{0}-\eta,t_{0}+\eta]\subset I$. We further fix a non-trivial smooth function $ \rho\in C_{c}^{\infty}((1,2))$, and, for all
$N\geq1$, define
\begin{equation}\label{eq:sharp-trace-data}
 \widehat f_{N}(r\theta) = e^{\mp   ict_{0}r}\rho(r/N)\theta_1,
 \qquad r>0,\quad \theta=(\theta_1,\theta_2)\in\mathbb{S}^{1}.
\end{equation}
Since $\widehat f_{N}$ is smooth and compactly supported in the
annulus $\{N<|\xi|<2N\}$, it is readily seen that   $f_{N}\in\mathcal{S}(\mathbb{R}^{2})$.  
Now, computing the polar Fourier inversion at $x=0$ yields that 
\begin{equation*} 
 \left(\partial_{1}e^{\pm ict|D|}f_{N}\right)(0)=iC_{\mathcal F}A 
 \int_{0}^{\infty}
 e^{\pm  ic(t-t_{0})r}r^{2}\rho\left(\frac{r}{N}\right)\,dr,
\end{equation*}
where $C_{\mathcal F}>0$ is an absolute constant that depends only on the Fourier-transform
normalization, and
\begin{equation*}
 A 
 \bydef  
 \int_{\mathbb{S}^{1}}\theta_{1}^{2}\,d\theta
 >0.
\end{equation*}  
Next, making the change of variables $r=Nq$ yields that 
\begin{equation*}
 \left(\partial_{1}e^{\pm ict|D|}f_{N}\right)(0)
 \sim i  N^{3}
 H\bigl(\pm  cN(t-t_{0})\bigr), 
\end{equation*}
where  
\begin{equation*}
 H(\tau)\bydef 
 \int_{0}^{\infty}e^{i\tau q}q^{2}\rho(q) dq.
\end{equation*}
It then follows that
\begin{equation*}
	\left\|
 \left.\partial_{1}e^{\pm  ict|D|}f_{N}
 \right|_{x=0}
 \right\|_{L^{2}_{t}(I)}^{2} \sim N^{5}
 \int_{\pm  cN(I-t_{0})}|H(\tau)|^{2}\,d\tau.
\end{equation*} 
Observe that the integration domain on the right-hand side contains
$[-cN\eta,cN\eta]$. Moreover, $H\in L^{2}(\mathbb{R})$ and
$H\not\equiv0$, by the one-dimensional Plancherel theorem. Hence,
for all sufficiently large $N$, one finds that
\begin{equation*}
 \int_{\pm cN(I-t_{0})}|H(\tau)|^{2}\,d\tau
 \geq
 \int_{-cN\eta}^{cN\eta}|H(\tau)|^{2}\,d\tau
 \geq
 \frac12\|H\|_{L^{2}(\mathbb{R})}^{2}.
\end{equation*}
Together with the corresponding   upper bound, we deduce that 
\begin{equation*} 
 \left\|
 \left.\partial_{1}e^{\pm  ict|D|}f_{N}
 \right|_{x=0}
 \right\|_{L^{2}_{t}(I)}
 \sim  N^{5/2}.
\end{equation*}

On the other hand, it is readily seen from
\eqref{eq:sharp-trace-data} that 
\begin{equation*}
 \begin{aligned}
 \|f_{N}\|_{\dot H^{s}}^{2}
 &\sim 
 A
 \int_{0}^{\infty}
 r^{2s+1}\left|\rho\left(\frac{r}{N}\right)\right|^{2}\,dr
 \\
 &= A
 N^{2s+2}
 \int_{1}^{2}q^{2s+1}|\rho(q)|^{2}\,dq
 \sim  N^{2s+2}.
 \end{aligned} 
\end{equation*}
All in all, putting the last two identities together yields the lower bound
\begin{equation*}
 \frac{
 \left\|
 \left.\nabla e^{\pm  ict|D|}f_{N}\right|_{x=0}
 \right\|_{L^{2}_{t}(I)}
 }{
 \|f_{N}\|_{\dot H^{s}}
 }
 \gtrsim N^{3/2-s},
\end{equation*}
which exhibits the claimed  growth as $N$ tends to infinity, whenever  $s<\nicefrac 32$. This completes the
proof.
\end{proof}

\section*{Acknowledgement} 

The author acknowledges support from the  BAGEP Award of the Science Academy, Turkey.

\bibliographystyle{plain}
\bibliography{references.bib}
 
\end{document}